\documentclass[a4paper,leqno,12pt,twoside]{amsart}
\usepackage{amsmath,amsfonts,amssymb,amsthm,amscd}
\usepackage[utf8]{inputenc}
\usepackage[english]{babel}

\usepackage[top=1.2in,bottom=1.2in,left=1.in,right=1.in]{geometry} 
\usepackage{graphicx}
\usepackage{paralist}
\usepackage{tabto}
\usepackage{standalone}
\usepackage{tikz}
\usetikzlibrary{matrix}
\usetikzlibrary{arrows}
\usepackage{xfrac}
\usepackage{tikz}
\usepackage{tikz-cd}

\usepackage[all]{xy}

\usepackage{enumerate}
\usepackage{xcolor}
\usepackage{aliascnt}

\newtheorem{teo}{Theorem}[section]
\newtheorem{defin}[teo]{Definition}
\newtheorem{prop}[teo]{Proposition}
\newtheorem{cor}[teo]{Corollary}
\newtheorem{lemma}[teo]{Lemma}

\theoremstyle{definition}

\newtheorem{remark}[teo]{Remark}

\newcommand{\sym}{{\operatorname{Sym}}}

\numberwithin{equation}{section}

\newcommand{\meno}{^{-1}}
\newcommand{\ra}{\rightarrow}

\newcommand{\restr}[1]          {\vert_{#1}}
\newcommand{\lra}{\longrightarrow}

\title[Generic Torelli for quintics]{Generic Torelli for coverings of plane quintics ramified in two points}

\dedicatory{Dedicated to the memory of Alberto Collino}

\author{J.C. Naranjo$^{1,2}$}
 \address{Juan Carlos Naranjo \newline 1. Departament de Matem\`atiques i Inform\`atica,
Universitat de Barcelona, Gran Via de les Corts Catalanes, 585, 08007 Barcelona, Spain \newline 2. Centre de Recerca Matemàtica, Edifici C, Campus Bellaterra, 08193 Bellaterra, Spain }
 \email{jcnaranjo@ub.edu}

 \author{I. Spelta}
 \address{Irene Spelta \newline Centre de Recerca Matemàtica, Edifici C, Campus Bellaterra, 08193 Bellaterra, Spain  }
 \email{ispelta@crm.cat}

\begin{document}
\pagestyle{myheadings}

\begin{abstract}
The aim of this paper to prove that the ramified Prym map restricted to the locus of coverings of quintic plane curves ramified in 2 points is generically injective.
\end{abstract}
\maketitle

\section{Introduction}

Given a finite morphism of smooth projective  irreducible curves, one defines the attached Prym variety as the component of the origin of the kernel of the induced map between their corresponding Albanese (Jacobian) varieties. By construction, this is a polarized abelian variety whose geometry reflects many of the properties of the finite map and, in many cases, of the Brill-Noether loci of the source and target curves. This theory is specially rich in the case of unramified double coverings. Let us denote by $\mathcal R_g$ the moduli of isomorphism classes of pairs $(C,\eta)$, where $C\in \mathcal M_g$ and $\eta $ is a non-trivial $2$-torsion point in the Jacobian of $C$. This is equivalent to an irreducible double cover of $C$. After the seminal work of Mumford (\cite{mu}), the Prym map between moduli stacks:
\[
\mathcal P_g:\mathcal R_g \lra \mathcal A_{g-1},
\]
sending $[(C,\eta)]$ to  its Prym variety $P(C,\eta)$,
has been extensively studied. Although is known to be generically injective and non-injective for $g\ge 7$ very few is known about its behavior when restricted to special subvarieties. It is worthy to  mention that Beauville extended $\mathcal P_g$ to a partial compactification $\overline {\mathcal R}_g$ allowing some coverings of stable curves (see \cite{be_invent}). The main property of this extended map $\overline {\mathcal P}_g$ is its properness.

More recently, double coverings ramified in an even number $r>0$ of points have been considered from the Prym map scope.
Let us denote by $\mathcal R_{g,r}$ the moduli of ramified coverings of curves of genus $g$ ramified in $r$ points. This classifies isomorphism classes of elements $(C,\eta, B)$, where $C\in\mathcal M_g$, $\eta \in Pic^{\frac r2}(C) $ and $B$ is a reduced divisor in $\vert \eta^{\otimes 2}\vert $.  Nagaraj and Ramanan studied the $r=4$ case in \cite{nagaraj_ramanan} and later Marcucci and Pirola considered the injectivity of
\[
\mathcal P_{g,r}:\mathcal R_{g,r} \lra 
\mathcal A_{g-1+\frac r2}^{\delta}
\]
in general, see \cite{mp} ($\delta $ is the type of the polarization, it is of the form $(1,\ldots ,1,2,\ldots,2)$). They proved the generic Torelli Theorem for most of the values of $(g,r)$. Finally, Ortega and the first author proved in \cite{naranjo-ortega2} the global Torelli Theorem for $r\ge 6$. Moreover, the fibers of $\mathcal P_{g,r}$ have been studied in  \cite{fns} when $\dim \mathcal R_{g,r}>\dim \mathcal A_{g-1+\frac r2}^{\delta }$.

The structure of the fibers of $\mathcal P_{g,2}$ and $\mathcal P_{g,4}$ is rather complicated, there exist positive dimensional fibers (e.g. if $C$ is hyperelliptic), the restriction to the tetragonal locus $\mathcal R_{g,r}^{tet}$ has degree $3$, etc. Notice that $\mathcal P_{g,2}$ can be seen as the restriction of the classical (unramified)  extended Prym map $\overline {\mathcal P}_g$ to a divisor of the boundary $\overline {\mathcal R}_g \setminus  \mathcal R_g$. 
A first step in the analysis of these maps is to consider the differential $d\mathcal P_{g,r}$. As proved in \cite{mp}, the codifferential is  given by the multiplication map 
\begin{equation}\label{codifferenziale}
	d\mathcal P_{g,r}^* (C, \eta, B): Sym^2 H^0(C, \omega_C \otimes \eta) \lra H^0(C, \omega_C^2 \otimes \mathcal O(B)).
\end{equation}

Using \cite[Remark 2.2]{naranjo-ortega2}, is easy to show that, if the differential $d\mathcal P_{g,r}$ is not injective at $[(C, \eta, B)]$, then:
	\begin{itemize}
		\item[1.]  $r=2$ and $\eta=\mathcal O_C(x+y-z)$ for $x,y,z\in C$ or $r=4$ and $h^0(C,\eta)>0$. Otherwise 
		\item[2.] $r=2$ and $C$ is hyperelliptic, trigonal or a quintic plane curve or $r=4$ and $C$ is hyperelliptic.
	\end{itemize}
Observe that in the first case the line bundle $\eta $ is special and in the second case the curve $C$ has Clifford index $\le 1$. This is not a characterization, apart from the hyperelliptic case we have no information on the rest of the possibilities. 

Our aim in this paper is to clarify the situation for double coverings of plane quintic curves with $r=2$. We  prove first the injectivity of the differential of $\mathcal P_{6,2}$ at a generic element:

\begin{teo}(Infinitesimal Torelli, Theorem \ref{differenziale iniettivo})
	Let $[(C, \eta, B)]$ be  a  general element in $ \mathcal{RQ}_{6,2}$. Then $d\mathcal P_{g,2}$ is injective at $[(C, \eta, B)]$.
\end{teo}

Moreover, we prove the generic Torelli Theorem for the Prym map restricted to the locus of quintic planes curves:

\begin{teo} (Generic Torelli) \label{thm_gen_tor}
The restriction of $\mathcal P_{6,2}$ to the quintic plane locus $\mathcal {RQ}_{6,2}$ is generically injective.
\end{teo}

Notice that the tetragonal construction (see \cite{do_tetr}) applies in this case: indeed, given $(C,\eta,B)\in \mathcal {RQ}_{6,2}$ we can identify the points of $B$ to get a nodal curve with $3$ natural tetragonal series. Then, there are other coverings in $\overline {\mathcal R}_7$ with the same Prym variety. The previous theorem implies that these other coverings do not belong to  $\mathcal {RQ}_{6,2}$ (as subspace of $\overline {\mathcal R}_7$).

Since the proof of the second theorem is set-theoretical, we can not deduce the first, since $\mathcal {RQ}_{6,2}$ could be contained in some ramification or singular locus in $\mathcal R_{6,2}$.

The structure of the paper is as follows: after some preliminaries on geometric properties of quintic plane curves (section 2), we  devote sections 3 and 4 to prove the two main theorems. In both cases we can convert the statement in some deformation problems: the study of extensions in $Ext^1(\omega_C\otimes \eta,\eta^{-1})$ with coboundary map of ranks $0$ and $1$.

\textbf{Acknowledgments:} We are very grateful to Paola Frediani, Mart\'i Lahoz and Gian Pietro Pirola for stimulating discussions on this subject. We thank Angela Ortega for pointing out an inaccuracy in a first version of the paper.

\section{Preliminaries}

\subsection{Plane quintics and conics}

It is an easy application of Riemann-Roch Theorem that a smooth plane quintic $C$ is neither hyperelliptic nor trigonal. Moreover, any $g^1_4$ on $C$ is obtained  from the $g^2_5$ by subtraction of a point, so the $W^1_4(C)$ is isomorphic to the curve $C$ itself (see \cite{acgh}, page 225). Here we collect some results describing the triplets $[(C,\eta, B=p_1+p_2)] \in \mathcal RQ_{6,2}$. The adjunction formula gives us: $\omega_C=\mathcal O_{\mathbb{P}^2}(-3+5) \restr{C}=\mathcal O_{C}(2)$. Thus, it turns out that conics, and specially conics tangent to $C$ in several points, play a crucial role in the analysis.

Our first result is the following:

\begin{lemma} \label{eta_two_minus_one}
 Let $[(C,\eta, B=p_1+p_2)] \in \mathcal RQ_{6,2}$. Assume that $p_1+p_2$ are general in $C$. Then $h^0(C,\eta \otimes \mathcal O_C(x))=0$ for all $x\in C$. In other words, $\eta $ is not of the form $\mathcal O(y+z-x)$, $x,y,z\in C$.
\end{lemma}
\begin{proof}
 By contradiction, assume that $\eta \cong \mathcal O(y+z-x)$. Then $2y+2z\sim 2x+p_1+p_2$ gives a $g^1_4$. So there is a point $q\in C$ and there exist lines $r_1$, $r_2$ such that $r_1\cdot C=2y+2z+q$ and $r_2 \cdot C=2x+p_1+p_2+q$. Hence, the line determined by $p_1$ and $p_2$ (that is, $r_2$) intersects $C$ in a point that belongs to a bitangent of $C$, this contradicts the genericity of the points. 
\end{proof}
When also the curve $C$ is general, we can go further. Indeed, we can state the following:
\begin{prop}\label{eta_three_minus_two} Let $[C,\eta, B]\in \mathcal{RQ}_{6,2}$ be general. Then $\eta $ is not of the form $\mathcal O_C(x_1+x_2+x_3-y_1-y_2)$. 
\end{prop}
\begin{proof}
 By contradiction, assume that $\eta \cong \mathcal O_C(x_1+x_2+x_3-y_1-y_2)$. Then:
 \[
 2 x_1 +2 x_2+ 2x_3 \sim 2 y_1 + 2y_2 +p_1 + p_2.
 \]
 Riemann-Roch formula tells us that $h^0(\Omega_C-(2 x_1 +2 x_2+ 2x_3))\geq1$, namely that there exists $z_1,z_2,z_3,z_4$ such that $2 x_1 +2 x_2+ 2x_3+ z_1+z_2+z_3+z_4$ and that $2 y_1 + 2y_2 +p_1 + p_2+z_1+z_2+z_3+z_4$ belong to $\vert \mathcal{O}_C(2)\vert$. This is equivalent to say that the pencil of conics passing through $z_1,z_2,z_3,z_4$ contains both a 3-tangent conic $Q_1$ and a 2-tangent conic $Q_2$ passing through $p_1,p_2$. 
 
 Let $C$ be a degenerate plane quintic given by 5 lines in general position: $C=\bigsqcup_{i=1}^5l_i$. Take a conic $Q_1$ 3-tangent in $x_1,x_2,x_3$. Let $R_4:=Q_1\cdot C-2x_1-x_2-x_3$. Let $\mathcal{Q}$ be the pencil of conics determined by $R_4$. By construction, the four points of $R_4$ will lie in two lines of $C$, let us assume $l_1,l_2$. If there existed a 2-tangent conic $Q_2\in\mathcal{Q}$ then $Q_2$ would be tangent to two of the lines of $Q$, assume $l_3,l_4$. Thus $Q_2\cdot l_5=p_1+p_2$. That would say that the 2 points $p_1,p_2$ would lie in a line of $Q$. By assumption of generality this is not true.
 
 Since the contradiction holds for this specific configuration we claim that the same still holds for the general element $[C,\eta, B]\in \mathcal{RQ}_{6,2}.$
\end{proof}

\subsubsection{Pluritangent conics}

The aim of this section is to prove Lemma \ref{one_section}, which is a technical result with important implications in the rest of the paper. The proof of this lemma is a consequence of rather elementary considerations on conics which we discuss now.

\begin{prop}\label{proposizione coniche}
Assume that $C$ is a general plane quintic, and
 let $\mathcal Q_k\subset \mathbb P^5$ be the closure of the set of conics of rank $\ge 2$ tangent to $C$ at $k$ different points. Then the codimension of $\mathcal Q_k$ in $\mathbb P^5$ is $k$ for $1\le k \le 5$.
\end{prop}
\begin{proof}
First, consider in $\mathbb P^5\times C^{(k)} $ the closed subset $I_k$ defined as the closure of the set of pairs  $(Q,p_1+\ldots +p_k)$ such that all the points $p_i$ are different and $Q$ is a conic tangent to $C$ at all $p_i$, that is $T_C(p_i)\subset T_Q(p_i)$.

Since all the fibers of $I_1 \longrightarrow C$ are projective spaces of dimension $3$ we get that $I_1$ is irreducible of dimension $4$. Since the fibers of $I_1 \longrightarrow \mathbb P^5$ are finite, we get that the image of $I_1$, which is $\mathcal Q_1$, is an irreducible hypersurface.  Notice that this hypersurface contains the Veronese surface of double lines $\mathcal V\subset \mathbb P^5$.

On the other hand, the generic fibers of  $I_2 \longrightarrow C^{(2)}$ are lines. As before, projecting on $\mathbb P^5$, we obtain the codimension 2 subvariety, $\mathcal Q_2$, parametrizing bitangent conics. Notice that also this variety contains $\mathcal V$.

The tritangent conics can be studied similarly but there are some differences with respect to the previous cases: the map into $C^{(3)}$ is no longer surjective and the preimage of the Veronese surface $pr_1^{-1}(\mathcal V)$ gives now a component of the locus we want to study. Define $T'\subset C^{(3)}$ to be the image of $I_3$. Notice that $T'$ contains $pr_2(pr_1^{-1}(\mathcal V))$, which is  the surface $\Gamma_3=\{D\in C^{(3)} \mid h^0(C,\mathcal O_C(1)(-D))>0\}$. Given $a,b\in C$, with different tangent lines, then $T'$ intersects the curve $a+b+C$ in a finite number  of points: the conics tangent to $C$ in $a$ and $b$ is a pencil providing a morphism $C\longrightarrow \mathbb P^1$ of degree $6$ with a ramification divisor of degree $22$. Thus, $T'\cdot (a+b+C)=22$ for all $a+b\in C^{(2)}$ and $T'$ has dimension $2$. Hence, we have a decomposition $T'=T\cup \Gamma_3$. This ends the case $k=3$. To go further in the study of $\mathcal Q_3$ we need the following Lemma, which is obvious.

\begin{lemma}
 Let $Q$ be a conic tangent to $C$ in  three different points $p_1, p_2, p_3$. Then $Q$ is degenerate if and only if at least one of these cases occurs:
 \begin{enumerate}
     \item [a)] the three points are in a line $r$, and $Q$ is the double line $r^2$,
     \item [b)] two of the points, say $p_1, p_2$, have the same tangent line $b$ (a bitangent) and $Q$ is formed by $b$ and the line tangent to $C$ in $p_3$,
     \item [c)] two of the points, say $p_1, p_2$, have the same tangent line $b$ (a bitangent) and $p_3 \in b$; in this case $Q$ is formed by $b$ and any line through $p_3$.
 \end{enumerate}
\end{lemma}

 We say we are in the ``bitangent case'' in the case c) of the lemma (this is the only case where the conic is not completely determined by the divisor $D$ of degree $3$). Moreover, observe that the case a) corresponds to the map $pr_1^{-1}(V)\ra \Gamma_3$ which  has degree $1$.
 In general, the fibers of   $I_3 \ra T'$ are  points, except in a finite number of points, the bitangent cases, where the fiber is isomorphic to $\mathbb P^1$. Let $\widetilde T $ be the blow-up of $T$ at the bitangent points.

Now we define a  map on the open set $U\subset T$ of the non-bitangent points. Denote by $Q_D$ the only conic such that $Q_D\cdot C \ge 2D$ for a $D\in U$. Then we consider the other points of intersection of the $Q_D$ with $C$:
\[
 \varphi: U\ra C^{(4)}, \quad D\mapsto Q_D\cdot C-2D.
\]
Notice that this map could also be defined in the other component
$\Gamma_3$. In this case the image is simply the small diagonal $\Delta_{2,2}=\{2x+2y\mid x,y\in C\}$. Moreover, observe  that the map on $T$ extends to $\widetilde T$: an element $(Q,x_1+x_2+x_3)\in \widetilde T$ is in the exceptional divisor if $Q$ is a couple of lines $r\cdot s$ such that $r$ is a bitangent with $r \cdot C=2x_1+2x_2+x_3$ and $s$ satisfies $s\cdot C=x_3+a_1+a_2+a_3+a_4$. Then $\varphi (Q,x_1+x_2+x_3)=a_1+a_2+a_3+a_4$.

Let us denote by $\Delta_4\subset C^{(4)}$ the big diagonal, that  is, the image of the addition $\Delta \times C^{(2)}\longrightarrow C^{(4)}$, where $\Delta \subset C^{(2)}$ is the usual diagonal. Observe that $\Delta_{2,2} \subset \Delta_4 $.
 
 Since the image of the exceptional divisors are clearly not contained in $\Delta_4$ (but with non-empty intersection with $\Delta_4$) we deduce that $T_4:= \varphi (\widetilde T)\cap \Delta_4$ has dimension $1$. This proves the case $k=4$.
 
 Finally, the case $k=5$  reduces to notice that there is a bijection between $\mathcal{Q}_5$ and the non-trivial 2-torsion points $\alpha \in JC$ with
$h^0(C,\mathcal{O}_C(1)\otimes\alpha) > 0$. Hence finiteness is given by the theory of theta characteristics. \\ \end{proof}

The following holds. 

\begin{lemma}\label{one_section} For a generic element $[(C,\eta, B)]\in \mathcal RQ_{6,2}$, we have that:
 \[
 h^0(C,\mathcal O_C(1)\otimes\eta)=1
 \]
\end{lemma}
\begin{proof} By Riemann-Roch Theorem, it is equivalent to show that $h^0(C,\mathcal O_C(1)\otimes\eta^{-1})=0$. Assume, by contradiction, that $h^0(C,\mathcal O_C(1)\otimes\eta^{-1})>0$, namely, that $\forall \eta$ there exist $x_1,x_2,x_3,x_4\in C$ such that $\eta (x_1+x_2+x_3+x_4)\cong  \mathcal O_C(1)$. This is equivalent to require that $\forall p_1,p_2\in C$ there exist $x_1,x_2,x_3,x_4\in C$ such that $ p_1+p_2 +2x_1+2x_2+2x_3+2x_4$ is the intersection of $C$ with a conic $Q$ which will be $4$-tangent to the quintic C. If the rank of $Q$ were $1$ then $p_1=p_2$ which is not possible. Hence, we have a $2$-dimensional family of conics of rank $\ge 2$ which are tangent at $4$ points of $C$. This contradicts Proposition \ref{proposizione coniche}.
\end{proof}

\subsection{Rank 2 vector bundles on curves}

The following result will be very useful along the paper:

\begin{lemma}{Beauville \cite[\S X.7]{beauBook}}\label{lemma beauville}
	Let $C$ be a curve and $E$ a rank 2 vector bundle on it. If \[2h^0(C, E)-3>h^0(C, \det E)\]
	then the kernel of the map $\alpha : \Lambda^2 H^0(C, E)\ra H^0(C, \Lambda^2 E )$ has a decomposable element defining (up to a saturation of the base locus) a line bundle $L\subset E$ such that $h^0(C, L)\geq 2$ and the quotient $E/L$ is a line bundle. 
\end{lemma}

\section{Infinitesimal Torelli for plane quintics}

The aim of this section is to prove the injectivity of  the differential of the Prym map in a general $[(C,\eta,B)]\in  \mathcal RQ_{6,2}$.

\begin{teo}\label{differenziale iniettivo}
	Let $[(C, \eta, B)]$ be  a  general element in $ \mathcal{RQ}_{6,2}$. Then $d\mathcal P_{g,2}$ is injective at $[(C, \eta, B)]$.
\end{teo}
	\begin{proof}
		By contradiction, assume there exists $\xi\in H^1(T_C(-p_1-p_2))$ such that
		\[
		\begin{aligned}
		 d\mathcal P_{g,2}: H^1(T_C(-p_1-p_2))&\ra S^2H^1(\eta\meno)=Hom^s(H^1(C,\eta ^{-1})^*, H^1(C,\eta ^{-1})) \\
		 \xi&\mapsto0
		\end{aligned}
		\]
		This means that $\xi\in Ext^1(\mathcal O_C, T_C(-p_1-p_2))$ corresponds to the exact sequence \[\xi: \;\; 0\ra T_C(-p_1-p_2)\ra E\ra \mathcal O_C\ra 0  \] and that 
		\[0\ra\eta\meno\ra E\otimes \omega_C\otimes\eta\ra \omega_C\otimes\eta\ra 0 \]
		has zero coboundary map. Thus, the study of the injectivity of the differential of the Prym map corresponds to the study of extensions $Ext^1(\omega_C\otimes\eta, \eta\meno)$ of rank zero. Putting $F:=E\otimes \omega_C\otimes\eta$, we get \[H^0(C,F)=H^0(C,\omega_C\otimes\eta)=6. \] 
		Furthermore, $\det F=\omega_C. $ Hence, by Lemma \ref{lemma beauville}, the kernel of the map \[\alpha: \Lambda^2H^0(C,F)\ra H^0(C,\Lambda^2 F)=H^0(C,\omega_C)\] contains at least a decomposable element defining a line bundle $L$ such that $h^0(C,L)\geq 2$ and $F/L$ is a line bundle. Line bundles $L\hookrightarrow F$ of such type are parametrized by points lying in the intersection of $\ker (\alpha)$ with $Grass(2, H^0(C,F))$. They fit in the following diagram:
		\begin{equation}\label{diagramma croce}
		\begin{tikzcd}
&	&0	\arrow[d] & &\\
&	&L \arrow{dr}{\tau}\arrow[d] & &\\
0\arrow[r]&\eta\meno\arrow[r]&F\arrow[r]\arrow[d]&\omega_C\otimes\eta\arrow[r]&0\\
& &M	\arrow[d] & &\\
& &0 & & 
		\end{tikzcd}
		\end{equation}
Notice that	the map $\tau$ is necessarily non-zero, therefore we always have that:
\begin{equation}\label{condizione tau }
	h^0(C,\omega_C\otimes \eta\otimes L\meno)>0.
\end{equation}

The vertical exact sequence of diagram \eqref{diagramma croce} gives us \begin{align}\label{condizioni}
			h^0(C,L)+h^0(C,M)&\geq 6\\
			h^0(C,L)-h^1(C,L)&=\deg L -5.\notag
		\end{align}
Since $M=\omega_C\otimes L\meno$, the sum of the two lines of \eqref{condizioni} gives us $2h^0(C,L)\geq \deg L +1.$ 
Moreover, since $C$ is non-hyperelliptic, Clifford's Theorem gives $\deg (L)>2(h^0(C,L)-1)$. Hence $\deg(L)+1\ge 2 h^0(C,L)$.  All together gives $5\le \deg(L)+1=2h^0(C,L)$. Since $h^0(C,L)\le 6$, the only possible values for $(\deg(L),h^0(C,L))$ are  $(5,3)$, $(7,4)$, $(9,5)$ and $(11,6)$. 

Let us analyze all these possibilities:
		\begin{enumerate}
		\item[1.] $ h^0(C,L)=3 $ and $\deg L= 5$. In this case, necessarily $L=\mathcal O_C(1)$. 
		\item[2.] $ h^0(C,L)=4 $ and $\deg L=7$. This is not allowed since  $\omega_C\otimes L\meno$ would be a $g^1_3$.
		\item[3.] $ h^0(C,L)=5 $ and $\deg L=9$. This implies that $L=\mathcal O_C(K_C-x)$, for some $x$ in $C$. As observed in \eqref{condizione tau }, diagram \eqref{diagramma croce} would require $0<h^0(C, \omega_C(\eta)\otimes L\meno)=h^0(\eta +x)$. Therefore, there would exist $y,z\in C$ such that $\eta +x\sim y+z$. This contradicts Lemma \ref{eta_two_minus_one}.
		\item[4.] $ h^0(C,L)=6$ and $\deg L=11$. This is not possible, since diagram \eqref{diagramma croce} requires $h^0(C,\omega_C(\eta)\otimes L\meno)>0$. 
		\end{enumerate} 
		
Hence, we can assume from now on that $L=\mathcal O_C(1)$ and therefore also $M=\mathcal O_C(1)$. By assumption, the coboundary map of the vertical exact sequence $\delta : H^0(C,\mathcal O_C(1))\ra H^1(C,\mathcal O_C(1)) $ is the zero map. The map $\delta$ is given by the cup product with $\varsigma\in Ext^1(\mathcal O_C(1),\mathcal O_C(1))\cong H^1(C,\mathcal O_C)\cong H^0(C,\omega_C)^*. $ We consider
the cup-product map:
\[
H^0(C,\mathcal O_C(1))\otimes H^1(C,\mathcal O_C ) \ra H^1(C,\mathcal O_C(1))
\]	
that induces
\[ 
H^0(C,\mathcal O_C(1)) \otimes H^0(C,\mathcal O_C(1))\otimes H^1(C,\mathcal O_C )\ra \mathbb{C}.
\]	
Hence we obtain that $\varsigma\cdot H^0(C,\omega_C)=0$, namely $\varsigma=0$.
This shows that the vertical exact sequence splits and thus the horizontal exact sequence becomes
\[
0\ra \eta\meno\ra \mathcal O_C(1)\oplus \mathcal O_C(1)\ra \omega_C\otimes 	\eta\ra 0.	
\] 

The non-zero section of $\mathcal O_C(1)\otimes\eta$ determines a skyscraper subsheaf (roughly speaking, the quotient $\mathcal O_C(1)/\eta\meno$) of $\omega_C\otimes 	\eta$. Since this is impossible we conclude the proof.  \end{proof}

\section{Generic Torelli in the locus quintic plane quintics}

In view of the results of the previous section it makes sense to ask whether the  restriction of $\mathcal P_{g,2}$ to $\mathcal RQ_{6,2}$ is generically injective. All this section is devoted to the proof of the generic Torelli Theorem \ref{thm_gen_tor}.

\begin{defin}
The \textit{semicanonical curve} $C_\eta$ is the image of $C$ through the projective map associated with the line bundle $\omega_C\otimes\eta$. 

\end{defin}
We denote with $I_2(C_\eta)$ the space of the homogeneous quadratic  polynomials vanishing on $C_\eta$, namely the kernel of the multiplication map \eqref{codifferenziale} (\cite{green_laz}). Our goal is to recover the curve $C_\eta$ in the intersection of such quadrics. Then $C_\eta$ identifies, up to isomorphism, the pair $(C,\eta)$. Since $C$ is non-hyperelliptic, then $h^0(C,\eta^{\otimes 2})=1$, and the divisor $B$ is completely determined by $C$ and $\eta$. 

We recall that there is a natural bijection between the points of $\mathbb{P}H^0(C,\omega_C\otimes \eta)^*$ lying in the intersection of the quadrics of $I_2(C_\eta)$ and the extensions (up to isomorphism and multiplication by a scalar)
\begin{equation}\label{estensioni}
	0\ra\eta\meno\ra E\ra \omega_C\otimes\eta\ra 0 
\end{equation}
	with coboundary map of rank 1. Moreover, given $p$ in the intersection of the quadrics, the image of the corresponding $\delta: H^0(C,\omega_C\otimes\eta)\ra H^1(C,\eta\meno)$ identifies a 1-dimensional space in $ H^0(C,\omega_C\otimes\eta)^*$ which corresponds exactly to $p\in \mathbb{P}H^0(C,\omega_C\otimes \eta)^* $. This is Lemma 1.2 of \cite{lange_sernesi}. 

So let us classify extensions \eqref{estensioni} with coboundary map of rank 1. This is equivalent to require $h^0(C,E)=5$, hence we have 
\[
2h^0(C,E)-3=7 > h^0(C,\det E)=h^0(C,\omega_C)=6.
\]
Thus, we can apply Lemma \ref{lemma beauville}: $E$ has a sub-line bundle $L$ such that $h^0(C,L)\geq 2$ and $E/L$ is a line bundle. The situation is summarized in a diagram as the one in \eqref{diagramma croce}. In the rest of the paper we consider two different subsets of rank 1 extensions in $Ext^1(\omega_C\otimes\eta, \eta\meno)$:
\[Ext^1_{ns}(\omega_C\otimes\eta, \eta\meno)\quad\text{and}\quad Ext^1_{s}(\omega_C\otimes\eta, \eta\meno). \]
The first consists of rank 1 extensions $E\in Ext^1(\omega_C\otimes\eta, \eta\meno)$  admitting non-special sub-line bundles $L\subset E$, while the second deals with the special cases.

\subsection{L non-special.} In this case, since $h^0(C,M)=h^0(C,\omega_C \otimes L^{-1})=h^1(C,L)$, we get that $h^0(C,L)=h^0(C,E)=5=h^0(C,\omega_C
\otimes \eta)-1 $. Notice that the map $L\ra \omega_C\otimes \eta$ is non-zero, therefore  $L= \omega_C\otimes \eta(-p) $ and so, in cohomology, we obtain the following diagram:
	\begin{equation}
	\begin{tikzcd}[column sep=1.2em]
	&	&0	\arrow[d] &0	\arrow[d] &\\ 
	&	&H^0(C,\omega_C \otimes \eta(-p))\arrow[d]\arrow[equal]{r} &H^0(C,\omega_C\otimes \eta(-p) )\arrow[d] & \\
	0\arrow[r]&H^0(C,\eta\meno)\arrow[r]&H^0(C,E)\arrow[r]\arrow[d]&H^0(C,\omega_C\otimes\eta)\arrow[r]\arrow[d]&H^0(C,\omega_C\otimes\eta)^*\\
	& &H^0(C,\eta\meno(p))\arrow[r]	\arrow[d] &\mathbb{C}_p\arrow[ur] &\\
	& &0. & & 
	\end{tikzcd}
	\end{equation}
	The fact that $ H^0(C,\omega_C\otimes \eta(-p))\ra H^0(C,\omega_C\otimes\eta)^*  $ is the zero map, allows us to consider the map $\mathbb{C}_p\ra H^0(C,\omega_C\otimes\eta)^* $, therefore with obtain the image of $p$ in $C_{\eta}$. 
		\begin{remark}
		Notice that the  vertical extension splits for every $p$. Indeed we have \[
	\begin{aligned}
	&Ext^1(\eta^{-1}(p),\omega_C\otimes \eta(-p))\cong Ext^1(\mathcal O_C,\omega_C\otimes \eta ^{\otimes 2}(-2p))\\ & \cong H^1(C,\omega_C\otimes \eta ^{\otimes 2}(-2p))\cong H^0(C,\mathcal O_C(2p-p_1-p_2))^*=0.
	\end{aligned}
	\]
	\end{remark}
	This remark suggests that, vice versa, the extension   
	\[
	0 \ra \eta ^{-1} \ra \eta^{-1}(p) \oplus \omega_C \otimes \eta (-p) \ra \omega_C \otimes \eta \ra 0. 
	\]
	corresponds exactly to the point $p\in\mathbb{P}H^0(\omega_C\otimes\eta)^*$. Thus, we have the following:
	\begin{prop}
	  The extensions in $Ext^1_{ns}(\omega_C\otimes\eta, \eta\meno)$ trace out the semicanonical curve $C_\eta$ in $\mathbb{P}H^0(\omega_C\otimes\eta)$.
	\end{prop}
	
\subsection{L special.} In this case, using diagram \eqref{diagramma croce} and the Riemann-Roch formula, we get conditions similar to \eqref{condizioni} and we obtain that $2h^0(C,L)-2 \deg(L) \ge 0$. Therefore 
\[
1\leq Cliff(L)=\deg(L)- 2 h^0(C,L)+2 \leq 2.
\]
Summarizing, the line bundle $L$ satisfies the following three conditions:
\begin{enumerate}
    \item [a)] $Cliff(L)=1,2$,
    \item [b)] $2\le h^0(C,L) \le h^0(C,E)=5$,
    \item [c)] $h^0(C,\omega_C\otimes \eta \otimes L^{-1})>0$.
\end{enumerate}

We need some notation to state the next lemma. By Lemma \ref{one_section} we have that $ h^0(C,\mathcal O_C(1)\otimes \eta)=1 $. Let $t$ be the generator and let $D$ be the unique effective divisor in the linear series $\vert \mathcal O_C(1)\otimes \eta \vert $. Put $D=a_1+\ldots +a_6$. We have the following:

\begin{lemma} Let $C$ be a general plane quintic and let $L$ be a line bundle satisfying conditions a), b), c). We have the following possibilities: 
\begin{enumerate}\label{casi possibili}
    \item  If $Cliff(L)=1$, then $L= \mathcal O_C(1)$.
    \item  If $Cliff(L)=2$, then either there exists $p\in C$ such that $L\cong \mathcal O_C(1)(-p)=:L_p$, or $L=\mathcal O_C(1)(a_i)=:M_{a_i}$, $i=1,\ldots ,6$.
 \end{enumerate}
\end{lemma}
\begin{proof}
Assume first that $Cliff(L)=1$, in particular $\deg(L)=2h^0(C,L)-1$. Using condition b), notice that $h^0(C,L)=2$ would give a trigonal series, and that $h^0(C,L)=3$ implies $L\cong \mathcal O_C(1)$. 
 On the other hand, $h^0(C,L)=4$, $\deg L=7$ implies that $\omega_C\otimes L\meno$ is a $g^1_3$. Finally, if $h^0(C,L)=5$ and $\deg L=9$, we have that $L=\omega_C(-p)$, for some $p\in C.$ Then, by the condition c),  $\eta \otimes \mathcal O_C(p)\cong \mathcal O_C( x_1+x_2)$, for some $x_1,x_2\in C$. This contradicts Lemma \ref{eta_two_minus_one}.
 
 Assume now that $Cliff(L)=2 $ and thus $ 2 h^0(C,L)=\deg L$. Again, using condition b), if $h^0(C,L)=2$, then $L$ is a $g^1_4$ and then is of the form $L_p$ as in the statement. Let us consider the cases $(h^0(C,L),\deg (L))=(3,6), (4,8),(5,10)$.

If $\deg L=6$ and $ h^0(C,L)=3$, then $\omega _C\otimes L^{-1}$ gives a $g^1_4$, hence $L=\mathcal O_C(1)(p)=:M_p$ for some $p\in C$. The condition $h^0(C,\omega_C(\eta)\otimes L\meno)=h^0(C,\mathcal O_C(1)\otimes \eta (-p))>0$ forces $p=a_i$.

In the case $h^0(C,L)=4$ and $\deg(L)=8$, we obtain that $\eta$ is of the form $\mathcal O_C(x_1+x_2+x_3-y_1-y_2)$ and this contradicts Proposition \ref{eta_three_minus_two}.

Finally, if $\deg L=10$ and $ h^0(C,L)=5$, then $h^0(C,L)-h^1(C,L)=5-h^1(C,L)=\deg(L)+1-6=5$. So $L$ is non-special.
\end{proof}

Let us study separately the cases in the Lemma \ref{casi possibili}. We will denote the three possibilities as $Ext^1_{s,g^2_5}(\omega_C\otimes\eta,\eta\meno)$, $Ext^1_{s,g^1_4}(\omega_C\otimes\eta,\eta\meno)$ and $Ext^1_{s,g^2_6}(\omega_C\otimes\eta,\eta\meno)$

\vskip 3mm
\noindent	\textbf{
Case $L=\mathcal O_C(1)$.}
\vskip 3mm

By Lemma \ref{one_section} we have $H^0(C,\mathcal O_C(1)\otimes \eta)=\langle t\rangle$. Tensoring the vertical sequence of \eqref{diagramma croce} with $\eta$ we get $H^0(C,\mathcal O_C(1)\otimes \eta)\hookrightarrow H^0(C,E\otimes\eta)$. Thus, using the section $t$, we can construct the following diagram:
 	\begin{equation}\label{diagramma croce g^2_5}
 \begin{tikzcd}
 &	&0	\arrow[d] &0\arrow[d] &\\
 &	&\mathcal O_C(1) \arrow[d]\arrow[equal]{r} &\mathcal O_C(1) \arrow[d]&\\
 0\arrow[r]&\eta\meno\arrow[r]\arrow[equal]{d}&E\arrow[r]\arrow[d]&\omega_C\otimes\eta\arrow[r]\arrow[d]&0\\
0 \arrow[r] &  \eta\meno \arrow[r] & \mathcal O_C(1)	\arrow[d]	\arrow[r] &\mathcal{O}_D \arrow[r] \arrow[d] & 0\\
 & &0 & 0& 
 \end{tikzcd}
 \end{equation}

Therefore, in cohomology we have:
	\begin{equation}\label{diagramma croce g^2_5 in coomologia}
 \begin{tikzcd}
 	&0	\arrow[d] &0	\arrow[d] &\\
 	&H^0(C,\mathcal O_C(1)) \arrow[d]\arrow[equal]{r} &H^0(C,\mathcal O_C(1)) \arrow[d]&\\
 0\arrow[r]&H^0(C,E)\arrow[r]\arrow[d]&H^0(C,\omega_C\otimes\eta)\arrow{r}{\delta}\arrow[d]&H^0(C,\omega_C\otimes\eta)^* \arrow[equal]{d}\\
 0\arrow[r] &H^0(C,\mathcal O_C(1))		\arrow[r] &H^0(C,\mathcal{O}_D)\arrow{r}{f} &   H^0(C,\omega_C\otimes\eta)^* \\
  & & & 
 \end{tikzcd}
 \end{equation}
 
Notice that Im$(\delta)\subset $Im $(f)$.  Since $f$ is simply the evaluation at the points in the support of $D$, we have that $\mathbb P(\text {Im} f)$ is the linear variety generated by the image of the points by the semicanonical map 
$\phi : C\lra \mathbb P H^0(C, \omega_C \otimes \eta )^*$, that is $\langle \phi(a_1),\ldots ,\phi(a_6) \rangle =:\langle \phi(D) \rangle $. Since $H^0(C,\omega _C\otimes \eta (-D))=H^0(C,\mathcal O_C(1))\cong \mathbb C^3$, we have that 
$\langle \phi(D) \rangle$ is a projective plane in $\mathbb P H^0(C,\omega_C\otimes\eta)^*$. Then the point corresponding to the rank $1$ horizontal extension lands in this plane, which does not depend on the extension. Thus we have shown the following:
\begin{prop}\label{immagini estensioni g^5_2}
  The rank 1 special extensions in $Ext^1_{s, g^2_5}(\omega_C\otimes\eta, \eta\meno)$ correspond to points in the plane $\langle\phi(D)\rangle\subset\mathbb{P}H^0(C,\omega_C\otimes\eta)^*$.
\end{prop}

\begin{remark}
We stress that extensions in $Ext^1_{s, g^2_5}(\omega_C\otimes\eta, \eta\meno)$ do exist. Indeed they come from non-splitting  extensions \[
 0\ra \mathcal O_C(1)\ra E\ra \mathcal O_C(1)\ra 0 
\] 
with coboundary map $\delta: H^0(C,\mathcal O_C(1))\ra H^1(C,\mathcal O_C(1))$ of rank 1. We know that these ones do exist since they are described by the elements or rank 1 in the image of map 
\[
Ext^1(\mathcal O_C(1),\mathcal O_C(1))\xrightarrow{m^*}Hom(H^0(C,\mathcal O_C(1)), H^1(C,\mathcal O_C(1)))
\]
i.e., dualizing to
\[
S^2H^0(C,\mathcal O_C(1))\xrightarrow{m}H^0(C,\omega_C)=H^0(C,\mathcal O_C(2)),
\]
by rank $1$ conics. This shows that they correspond to points of the Veronese surface $\mathcal{V}\subset \mathbb{P}^5$.
\end{remark}

\noindent \textbf{Case $L_p=\mathcal O_C(1)(-p)$, $p\in C$.}
\vskip 3mm

We have that $L_p$ is a $g^1_4$. Let us denote with $M_p=\mathcal O_C(1)(p)$, $p\in C$. We consider extensions  \begin{equation}\label{estensione g^1_4}
\quad	0\ra L_p\ra E\ra M_p\ra 0
\end{equation}
with $h^0(C,E)=5$.
Observe that $h^0(C,L_p)=2$ and $h^0(C,M_p)=3$. Therefore the coboundary map 
\[
\delta: H^0(C,M_p)\ra H^1(C,L_p)\cong H^0(C,M_p)^*
\]
is the zero map.  Let us denote by $Ext^1_0(M_p,L_p)\subset Ext^1(M_p,L_p)$ the set of the rank 0 extensions as in \eqref{estensione g^1_4}. 

\begin{lemma}\label{estensioni verticali rango 0}
Points in $Ext^1_0(M_p,L_p)$ are in bijection with the co-kernel of the multiplication map:
\[
m: H^0(C,M_p)\otimes H^0(C,M_p)\ra H^0(C,\omega_C(2p)).
\]
\end{lemma}
\begin{proof} Consider an extension as in \eqref{estensione g^1_4}. Then the rank is $0$ if
taking $\omega \in H^0(C,M_p)$, we have $\xi\omega(\omega')=\xi(\omega\omega')=0$ for every $\omega'\in H^0(C,M_p)$.
This determines (and is determined by) an element in the co-kernel of the map: 
\[
m: H^0(C,M_p)\otimes H^0(C,M_p)\ra H^0(C,\omega_C(2p))\cong H^1(C,\mathcal O_C(-2p))^*\cong Ext^1(M_p,L_p)^*.
\] 
\end{proof}

Notice that in this case the cokernel is always non-empty. Indeed,  $H^0(C,M_p)=H^0(C,\mathcal O_C(1))$ for every $p$, hence the image of the multiplication map has co-dimension $1$ in the target, that is the codimension of $H^0(C,\omega_C)$ in $H^0(C,\omega_C(2p))$. Thus, up to scalars, the non-trivial rank 0 vertical extension in $Ext^1(M_p,L_p)$ is unique and it only depends on the point $p$. This says that the vector bundle $E$ is determined by $L_p\subset E$ for every point $p\in C$. For this reason, from now on we denote it by $E_p$. Furthermore, we can interpret $\mathcal{E}:=Ext^1_0(M_p,L_p)$ as a line bundle over the curve $C$. 

In order to present it more precisely, we discuss Lemma \ref{estensioni verticali rango 0} in families. Let $C\times C\xrightarrow{\pi_i}C$ be the projection on the $i$-th factor, $i=1,2$, and let $\Delta\subset C\times C$ be the diagonal. The sheaves \[\mathcal{F}:= \pi_2^*(\mathcal{O}_C(1))\otimes \mathcal{O}_{C\times C}(\Delta) \quad\text{and}\quad\mathcal{G}:= \pi_2^*(\omega_C)\otimes \mathcal{O}_{C\times C}(-2\Delta)\]
and the co-kernel of the map \[\sym^2 R^0\pi_{1*}(\mathcal{F})\lra R^0\pi_{1*}(\mathcal{G}) \]
yields the definition of $\mathcal{E}$. Notice that, by construction, we have 
\[\mathbb{P}(\mathcal{E})\cong C.\]

Let us fix a point $p\in C$ and let us consider the exact sequence \eqref{estensione g^1_4}. Tensoring with $\eta$, we deduce that
\[
h^0(C,E_p\otimes\eta)\cong h^0(C,M_p\otimes \eta)=2. 
\]

Therefore, for every $E_p$, we have a 2-dimensional family of diagrams parametrized by sections $s\in H^0(C,M_p\otimes \eta)$
as follows (here we put $(s)_0=D_s$): 
\begin{equation}\label{diagramma croce g^1_4}
 \begin{tikzcd}
 &	&0	\arrow[d] & &\\
 &	&L_p \arrow[d]\arrow[equal]{r} &L_p \arrow[d]&\\
 0\arrow[r]&\eta\meno\arrow{r}{\cdot s}\arrow[equal]{d}& E_p\arrow[r]\arrow[d]&\omega_C\otimes\eta\arrow[r]\arrow[d]&0\\
 0\arrow[r]&\eta\meno\arrow{r}{\cdot s} &M_p	\arrow[d]\arrow[r] &\mathcal O_{D_s}\arrow[r] &0\\
 & &0 & & 
 \end{tikzcd}
\end{equation}
 In other words, we have a map:
\begin{align}\label{fibrato rango 2}
    Ext^1_{s,g^1_4}(\omega_C\otimes\eta, \eta\meno)&\ra\mathcal{E}\\
    (s,p)&\mapsto <E_p>\notag
\end{align}
   with fibers given by $H^0(M_p\otimes \eta)$.

Notice that $\mathbb{C}\cong H^0(C,\mathcal O_C(1)\otimes \eta)\subset H^0(M_p\otimes \eta) $. In particular, the divisor $D+p$ belongs to $ \vert M_p\otimes\eta\vert$. Let $s=t\mu_p$, where $\mu_p(p)=0$, be the section such that $(s)_0=D+p$. We have the following:
  
   \begin{prop}\label{curva C nella superficie}
     The image of extension  $(t\mu_p,p)\in Ext^1_{s,g^1_4}(\omega_C\otimes\eta, \eta\meno)$ is the point $p$ in $C_\eta\subset \mathbb{P}H^0(C,\omega_C\otimes \eta)^*$ for every $p\in C$.
   \end{prop}
   \begin{proof}
        Let us consider the element $tl\in H^0(\omega_C\otimes\eta)$: $l$ is a line in $ H^0(C,\mathcal{O}_C(1))$ not passing through $p$. Taking diagram \eqref{diagramma croce g^1_4} in cohomology, one easily checks that the vertical sequence 
\[
0\ra H^0(C,L_p)\ra H^0(C,\omega_C\otimes\eta) \xrightarrow{v}H^0(C,\mathcal O_{D_s})\ra H^1(C,L_p)\ra 0,
\]
where $v$ is the evaluation map $\alpha\mapsto (\alpha(a_1), \dots, \alpha(a_6),\alpha(p)), $ dualizes to the horizontal sequence 
\[
\begin{aligned}
0\ra H^0(C,M_p)\ra H^0(C,\mathcal O_{D_{s}})&\xrightarrow{v^*} H^1(C,\eta\meno)\cong \\
& \cong H^0(C,\omega_C\otimes\eta)^*\ra   H^1(C,M_p)=H^0(C,L_p)^*\ra 0.
\end{aligned}
 \]
Since we know that the coboundary map $\delta : H^0(C, \omega_C\otimes \eta)\lra H^0(C, \omega_C \otimes \eta)^*$ has rank 1, once we show that $\delta(tl) \neq 0$, we are almost done. This follows immediately from the fact that 
\[
\delta(tl)=v^*v(tl)=v^*(0,\dots, 0,tl(p)).
\]
Up to a scalar, this yields the evaluation map in $p$, i.e. the point $p\in C_\eta$. \\
   \end{proof}
   \begin{remark}
   Notice that the previous proposition also clarifies how rank 1 extensions in $Ext^1_{s,g^1_4}(C,\omega_C\otimes\eta, \eta\meno)$, with $s\neq t\mu_p$, contribute to the intersection of the quadrics in $\mathbb{P}I_2(C_\eta)$. Indeed, it shows that changing the section $t\mu_p$ to a section $s\in H^0(C,\mathcal{O}_C(1)(p)\otimes \eta)$ in a neighborhood of it, by continuity, $\delta(tl)$ is again different from zero. This says that to study the image of the map 
\[\mathbb{P}H^0(C,\omega_C\otimes\eta)\otimes \mathbb{P}Ext^1_{s,g^1_4}(\omega_C\otimes\eta, \eta\meno)\ra \mathbb{P}H^0((C,\omega_C\otimes\eta)^*)\]
(which gives us all the points in $\mathbb{P}H^0(C,\omega_C\otimes\eta)^*$ coming from $\mathbb{P}Ext^1_{s,g^1_4}(\omega_C\otimes\eta, \eta\meno)$) 
it is sufficient to look at the image of 
\begin{equation}\label{mappa moltiplicazione finale}
   \phi: \{tl\}\otimes \mathbb{P}Ext^1_{s,g^1_4}(\omega_C\otimes\eta, \eta\meno)\ra \mathbb{P}H^0((C,\omega_C\otimes\eta)^*).
\end{equation}
  By \eqref{fibrato rango 2}, $S_1:=\mathbb{P}Ext^1_{s,g^1_4}(\omega_C\otimes\eta, \eta\meno)$ is a $\mathbb{P}^1$-bundle over $C$ with fibres $\mathbb{P}H^0(C,\mathcal{O}_C(1)(p)\otimes \eta)$. Thus, the problem finally turns upon the study of the surface $S$ which is  $\phi (S_1)$.
 \end{remark}
 
The following general fact holds:
\begin{prop}\label{two_surfaces}
Let $f:S_1\ra C$ be a $\mathbb{P}^1$-bundle over a base curve $C$ with section $\sigma:C \lra S_1$, we put $C_1:=\sigma(C)$. Let $\tau: S_1\ra \mathbb{P}^N$ be a morphism and set $S_2:=\tau(S_1)$ the image. Assume that $\dim S_2=2$, $\tau(\mathbb{P}^1)$ is a line and $\tau_{\vert C_1}$ is an embedding. Then $\tau$ is birational and  there is at most one point $q$ such that $S_1\setminus \tau \meno (\{ q\})$ is isomorphic to $S_2\setminus \{q\} $.
\end{prop}
\begin{proof}
Let $r_1$ be a line in $S_1$. First notice that $\tau\meno(\tau(r_1))$ cannot contain more than one vertical line (namely $r_1$), neither $\lambda r_1$, $\lambda\in \mathbb{C}$. This is because in these cases, the assumption on $\tau_{\vert C_1}$ would be violated. Therefore, assume now, by contradiction, that $\tau\meno(\tau(r_1))$ contains $r_1$ plus a curve $\Gamma$ not containing more lines and such that the map to $C$ has degree $\ge 2$. This is impossible since the points in the intersection $l\cap\ \Gamma$, with $l$ any other vertical line on $S_1$, must be mapped to $\tau(r_1)$. Thus, the image $\tau(l)$ can no longer be a line. Therefore, $\tau\meno(\tau(r_1))$ at worst is given by $r_1$ plus a section $\beta$. If this section is contracted to a point in $\tau(r_1)$ then we would obtain a cone and out of the vertex we are fine. Since lines transform into lines  there is at worst one point where this can happen, therefore we are done. On the other hand, if the section is not contracted to a point, then $\tau(r_1)$ would be birational to $\tau(C_1)$ and this is impossible.
\end{proof}
\begin{cor}
The morphism $\phi$ satisfies the hypothesis of Proposition \ref{two_surfaces}.
\end{cor}
\begin{proof}
 We already know that $\mathbb{P}Ext^1_{s,g^1_4}(\omega_C\otimes\eta, \eta\meno)$ is a $\mathbb{P}^1$-bundle over $C$. Moreover, let us observe that the bundle $Ext^1_{s,g^1_4}(\omega_C\otimes\eta, \eta\meno)$ naturally carries a section which is isomorphic to the curve $C$. It is obtained taking the point $\{t\mu_p\}\in \mathbb{P}H^0(\mathcal{O}_C(1)(p)\otimes \eta)$ for every $p\in C$. By Proposition \ref{curva C nella superficie}, this section is clearly sent by $\phi$ to the curve $C_\eta\subset S$. Finally, since $\phi$ is a multiplication map, the vertical line of the bundle are sent to line in $S$ and furthermore $\dim \text{Im}(\phi)=2$. 
\end{proof}
Thus, we can conclude with the following:
\begin{prop}
The rank 1 special extensions in $Ext^1_{s,g^1_4}(\omega_C\otimes\eta, \eta\meno)$ correspond to a surface $S\subset \mathbb{P}H^0(\omega_C\otimes \eta)^*$ birational to a ruled surface with base curve $C$.
\end{prop}

\vskip 3mm
\noindent \textbf{Case $L=\mathcal{O}_C(1)(a_i)=:M_{a_i}, i=1,\dots,6.$} 
\vskip 3mm

We consider extensions \[0\ra M_{a_i}\ra E\ra \omega_C \otimes M_{a_i}\meno\ra 0\]
with $h^0(E)=5$ and null coboundary map. An easy reformulation of Lemma \ref{estensioni verticali rango 0} shows that such extensions are in bijection with the co-kernel of the multiplication map 
\[m: H^0(C, \omega_C\otimes M_{a_i}\meno)\otimes H^0(C, \omega_C\otimes M_{a_i}\meno)\ra H^0(C, \omega_C(-2a_i)). \]
Since the co-kernel has dimension 1, it turns out that the rank 0 vertical extension is unique for every $a_i$, up to scalars. Thus, we have to study just six extensions, let us denote them $E_{a_i}.$ 

In order to produce a diagram like \eqref{diagramma croce g^1_4}, we observe that $h^0(C,E\otimes \eta)=3$: 2 generators come from $H^0(C,\mathcal{O}_C(1)\otimes \eta(a_i))$ and one from $H^0(C,\mathcal{O}_C(1)\otimes \eta(-a_i))$. The first two are not admissible since they would determine a skyscraper subsheaf of $\omega_C\otimes\eta$. Thus, it only remains to consider the section of $H^0(C,\mathcal{O}_C(1)\otimes \eta(-a_i))$. Notice that we already know the divisor in this linear series: it is $\hat{D_i}:=a_1+\dots+\hat{a}_i+\dots+a_6$. In short, we have to consider just 6 diagrams as the following:
\begin{equation}\label{diagramma croce g^2_6}
 \begin{tikzcd}
 &	&0	\arrow[d] & &\\
 &	&M_{a_i} \arrow[d]\arrow[equal]{r} & M_{a_i}\arrow[d]&\\
 0\arrow[r]&\eta\meno\arrow{r}\arrow[equal]{d}& E_{a_i}\arrow[r]\arrow[d]&\omega_C\otimes\eta\arrow[r]\arrow[d]&0\\
 0\arrow[r]&\eta\meno\arrow{r} & \omega_C\otimes M_{a_i}\meno	\arrow[d]\arrow[r] &\mathcal O_{\hat{D_i}}\arrow[r] &0\\
 & &0 & & 
 \end{tikzcd}
\end{equation}
Passing in cohomology, notice that the situation is very close to the one of diagram \eqref{diagramma croce g^2_5 in coomologia}. The point corresponding to the rank 1 horizontal extension in \eqref{diagramma croce g^2_6} lands in the plane $\langle\phi(\hat{D_i})\rangle$, which is exactly the same plane of Proposition \ref{immagini estensioni g^5_2}. Indeed, $H^0(C, \omega_C\otimes\eta(-\hat{D}))= H^0(C,\mathcal{O}_C(1)(a_i))=H^0(C,\mathcal{O}_C(1))\cong \mathbb{C}^3$. Since this is independent of $i$, we have the following:
\begin{prop}
The rank 1 special extensions in $Ext^1_{s,g^2_6}(\omega_C\otimes\eta,\eta\meno)$ correspond to 6 points in the plane $\langle\phi(D)\rangle\subset\mathbb{P}H^0(\omega_C\otimes \eta)^*$.
\end{prop}

To resume, our analysis shows what follows.
\begin{teo}
Let $(C,\eta, B)\in \mathcal{RQ}_{6,2}$ be a general element and let $\mathbb{P}I_2(C_\eta)$ be the space of the quadrics vanishing on the semicanonical model $C_\eta$ of $C$. Then the intersection of the quadrics consists of points landing on a projective plane and of a surface $S$. Moreover, there is a birational map from  a ruled surface with base curve $C_\eta$ to $S$ which is an isomorphism out of the preimages of a finite number of points in $S$. 
\end{teo}

Using this result now we can finish the proof of our main theorem \ref{thm_gen_tor}. Indeed, by using the differential of the Prym map at a generic point $(C,\eta, B)$ we obtain a family of quadrics whose intersection is, up to a finite number of points and a linear variety of dimension $2$, a surface $S$ as in the Theorem above. The curve $C_{\eta}$ is contained in this surface. Assume that there is another element in the fibre $(C',\eta',B')$. Then the curve $C'_{\eta '}$ is also contained in $S$. The preimage by $\phi $ of this curve has only one component  $C''$ birational to $C'_{\eta '}$. Then the map $C''\lra C$ must be birational and therefore $C\cong C'$. 
Moreover, the surface $S$ is either isomorphic to the ruled surface $S_1=\mathbb{P}Ext^1_{s,g^1_4}(\omega_C\otimes\eta, \eta\meno)$ or it is a cone with base $C_{\eta}$. In both cases the curve $C'_{\eta'}$ projects to $C_{\eta}$ (following the rulling in the first case or projecting from the vertex in the second). Hence this isomomorphism send $\eta '$ to $\eta$.  

\bibliographystyle{amsalpha}

\end{document}